\documentclass[11pt]{article}
\usepackage[english]{babel}
\usepackage{amsfonts,amsmath,amsxtra,amsthm,amssymb,latexsym}
\usepackage[cp866]{inputenc}

\makeatletter\@addtoreset {equation}{section}\makeatother
\theoremstyle{plain}

\newtheorem{theorem}{Theorem}[section]
\newtheorem{corollary}{Corollary}[section]
\newtheorem{lemma}{Lemma}[section]
\newtheorem{proposition}{Proposition}[section]
\newtheorem{definition}{Definition}[section]
\newtheorem{remark}{Remark}[section]

\renewcommand{\Re}{\mathop{Re}}
\renewcommand{\Im}{\mathop{Im}}

\newcommand{\ran}{\mathop{Ran}}
\newcommand{\dom}{\mathop{Dom}}
\renewcommand{\ker}{\mathop{Ker}}

\renewcommand{\L}{\mathcal{L}}
\newcommand{\W}{\mathrm{W}}

\def \rn{{\mathbb R}}
\def \nn{{\mathbb N}}
\def\R{\mathbb R}
\def\C{\mathbb C}

\def\Z{\mathbb Z}

\def\wt{\widetilde}
\def\la{\lambda}
\def\uph{\upharpoonright}
\def\H{\mathcal H}
\def\DD{\mathbb{D}}
\def\ep{\varepsilon}
\def\b{\mathfrak{b}}

\def \supp{\mathop{\rm supp}}
\def\h{h}
\def\fin{\mathrm{fin}}
\def\loc{\mathrm{loc}}
\def\I{\mathcal{I}}

\def\ux{\mathfrak{u}}
\def\one{\upharpoonleft \!\! \parallel}
\def\uph{\upharpoonright}
\def\wt{\widetilde}
\def\Sclass{\mathfrak{S}}
\def\matA{\rho}

\begin{document}

\title{\bf On the nature of ill-posedness of the forward-backward heat equation.}

\author{Marina Chugunova \\ {\small Department of Mathematics, University of
Toronto, Canada} \\
Illya M. Karabash \\ {\small Department of
Mathematics and Statistics, University of Calgary, Canada } \\
{\small Department of PDE, Institute of Applied Mathematics and Mechanics, Donetsk, Ukraine} \\
Sergei G. Pyatkov \\ {\small  Ugra State University, Hanty-Mansiisk,
Russia } \\ {\small Sobolev Institute of Mathematics, Novosibirsk,
Russia}}

\date{\today}
\maketitle

\begin{abstract}
We study the Cauchy problem with periodic initial data for the
forward-backward heat equation  defined by the J-self-adjoint linear
operator L depending on a small parameter. The problem has been
originated from the lubrication approximation of a viscous fluid
film on the inner surface of the rotating cylinder. For a certain
range of the parameter we rigorously prove the conjecture, based on
the numerical evidence, that the set of eigenvectors of the operator
$L$ does not form a Riesz basis in $\L^2 (-\pi,\pi)$. Our method can
be applied to a wide range of the evolutional problems given by
$PT-$symmetric operators.
\end{abstract}

\section{Introduction}
Analysis of the dynamic of a thin film of liquid which is entrained
on the inside of a rotating cylinder is of great importance in lots
of applications. For example when liquid thermosetting plastic is
placed inside a rotating mould the best quality can be achieved if
distribution of the liquid is as uniform as possible. More details
about this application can be found in \cite{rotmodel}. The same
problem arises in the coating of fluorescent light bulbs when
suspension consisting of a coating solute and a solvent is placed
inside a spinning glass tube. The model for the coating was
described for example in \cite{light model}.

The lubrication approximation is used extensively to study flows in
thin films. Under the assumption that the film is thin enough for
viscous entrainment to compete with gravity, the time evolution of a
thin film of liquid on the inner surface of a rotating in a gravity
field cylinder can be described by the forward-backward heat
equation:
\begin{equation}
\label{heat PDE}
 h_t + L h=0,\  \ \ \theta \in (-\pi,\pi), \quad t\in (0,T),
\end{equation}
where
\begin{equation}
\label{heat operator}
 L h =\varepsilon \, \partial_\theta (\sin \theta \, h_\theta)+ h_\theta, \quad
 h(-\pi)= h(\pi), \quad \varepsilon > 0.
\end{equation}
The effect of the surface tension is neglected in this linearized
model derived by Benilov, O'Brien and Sazonov in \cite{Benilov1}.

We prove that the related to this equation Cauchy problem
\begin{equation}
\label{Cauchy problem} \ h|_{t=0}=h_0(\theta), \quad
h(-\pi,t)=h(\pi,t).
\end{equation}
does not have a weak in the Sobolev sense solution $h(\theta,t)$
even locally in time if $h_0 (\theta)$ belongs to the class of
finitely smooth functions with $ \supp h_0\cap(\delta,\pi - \delta)
\neq \varnothing$.

The statement above can be roughly understood from the classic
theory of the parabolic equations that states that regularity of a
generalized solution depends on the regularity of the equation
coefficients ( in our case all coefficients are in
$C^{\infty}(-\pi,\pi)$) and from the time-reversibility of the
equation, i.e simultaneous change of the time variable $t$ to  $-t $
and the space variable $x$ to $-x$ leads to the same partial
differential equation. Time-reversibility and infinite regularity
generally imply ill-posedness.

The physical explanation of this explosive blow up of solutions is
related to a drop of fluid that can be detached from the film in the
upper part of cylinder, where the effect of the gravity is the
strongest \cite[p. 217]{Benilov1}.

The eigenvalues of the operator $L$ were studied asymptotically,
with application of the modified WKB approximation and numerically,
with application of the analytic continuation method, by Benilov,
O'Brien and Sazonov \cite{Benilov1} and they came to a very
interesting set of hypotheses: all eigenvalues of the operator $L$
are located on the imaginary axis, they are all simple and the set
of eigenfunctions is complete in $\L^2 (-\pi, \pi)$ that is not
typical for the ill-posed time-evolution problem.

The analysis of the spectral properties of this operator was
continued by Chugunova, Pelinovsky \cite{ChugPel} and by Davies
\cite{D07}. Using different approaches they analytically justified
that if the parameter $|\varepsilon| < 2$ then the operator is well
defined in the sense that it admits closure in $\L^2 (-\pi, \pi)$
with non-empty resolvent set. Analyzing tridiagonal matrix
representation of the operator $L$ with respect to the Fourier
basis, Davies \cite{D07} showed that $L$ admits an orthogonal
decomposition with respect to three invariant subspaces $\H^{2,0}
(\DD)$, $\H^{2,0} (\bar{\C} \setminus \DD) $, and $\ker(L)=\{c\one ,
\ c \in \C\} $ (see Section \ref{s sp prop} below) and used this
fact to prove that the nontrivial part $\wt L:=L \uph \H^{2,0} (\DD)
\oplus \H^{2,0} (\bar{\C} \setminus \DD)$ of $L$ has a compact
inverse of the Hilbert-Schmidt type. Therefore the spectrum of the
original operator $L$ is discrete with the only possible
accumulation point at infinity.

Under the additional condition ($1/ \varepsilon$ is not integer) it
was proved by Weir \cite{Weir1} that if there exists an eigenvalue
$\lambda$ of the operator $L$, then $\mu = i \frac{2
\lambda}{\varepsilon}$ is an eigenvalue of some symmetric operator,
hence $\lambda$ can be only pure imaginary. The elegant proof is
based on 
the continuation of the eigenfunctions into the Hardy space $\H^2
(\DD)$ in the unit disk $\DD = \{ z \in \C: |z| <1\}$. In this paper
we sharpen this result showing that the additional condition above
can be omitted. Note also that the similar problem was studied in
the recent preprint \cite{PT-sym} for a class operators that
includes the operator $L$.

It was shown numerically  \cite{ChugPel, D07} that the angle between
the subspace spanned by the $N$-first eigenfunctions and the
$(N+1)$-th eigenfunction of the operator $L$ tends to $0$ as $N$
goes to infinity. This gives the numerical evidence that the
eigenfunctions do not form a Riesz basis in $\L^2 (-\pi,\pi)$
because the related projectors are not uniformly bounded. One of the
main goals of this paper is to prove analytically this numerical
conjecture justifying that the operator $L$ is not similar to a
self-adjoint.

We also prove that the system of eigenvectors 
$L$ is complete in $\L^2 (-\pi,\pi)$ (see Theorem \ref{t Lcomp} and Remark \ref{r EigCompl}). Hence, this
implies that $L$ has infinite number of pure imaginary eigenvalues
that accumulate to $\pm i \infty$. As a consequence, due to the
linearity, the original Cauchy problem has infinitely many global in
time solutions which are linear combinations of harmonics $e^{i
\lambda_n t}u_{\lambda_n}(x)$ where $i \lambda_n$ is an eigenvalue
of the operator $L$ and $u_{\lambda_n}(x)$ is the related
eigenfunction.

The operator $L$ is $J$-self-adjoint in the Krein space with
$J(f(\theta)) = f(\pi - \theta)$ and therefore it belongs to the
class of $PT-$symmetric operators. Interesting development of the
spectral theory of $PT-$symmetric operators which are not similar to
self-adjoint ones can be found in \cite{LanTret, Shin1, Shin2,PT-sym}.

\textbf{Notations:} In the sequel, $C_1$, $C_2$, \dots denote
constants that may change from line to line but remain independent
of the appropriate quantities. We also use $\h'$, $\partial_{\theta}
h$, and $h_\theta$ for $\frac{d \h}{d \theta}$. The symbol $\one$
denotes the function that identically equals $1$ for $\theta \in
(-\pi,\pi)$. Let $T$ be a linear operator in a Hilbert space $H$.
The following classic notations are used: $\,\dom (T)$, $\,\ker
(T)$, $\,\ran (T)$ are the domain, the kernel, and the range of $T$,
respectively; $\sigma(T)$ and $\rho (T)$ denote the spectrum and the
resolvent set of $T$; $\sigma_p (T)$ stands for the set of
eigenvalues of $T$. We write
$f(x)\asymp g(x)\ \ (x\to x_0)$ if both $f/g$ and $g/f$ are bounded
functions in a certain neighborhood of $x_0$. By $\DD = \{ z \in \C:
|z| <1 \}$ we denote the open unit disc  in $\C$.


\section{Analysis of the differential equation. \label{s DE}}
One of the most general linear second-order differential equations
with periodic coefficients that can be solved by a trigonometric
series with a three-term recursion relation between the coefficients
was studied by Magnus and Winkler \cite{Heun} and has the form
\begin{equation}
\label{MW} (A + B \cos(2 \theta)) \frac{d^2 y}{d \theta^2} + C\sin(2
\theta) \frac{dy}{d \theta} + (D + E \cos(2 \theta))y = 0,
\end{equation}
where $A$, $B$, $C$, $D$, $E$ are constants. Under the additional
condition that the coefficient $A + B \cos(2 \theta)$ does not have
zeros located on the real axis they studied existence of the
periodic solutions to (\ref{MW}).

In this section we study the basic properties of the differential
equation  $\ell[\h](x)= f (x)$ where the differential expression $\h$ is
given by
\begin{equation}
\label{diffoperation} \ell [\h]:= \ep \frac {d}{d \theta} \left(
\sin (\theta) \frac {d\h }{d\theta}\right) + \frac{d \h}{d \theta} ,
\quad \theta \in (-\pi,\pi).
\end{equation} and use these properties to define the maximal periodic
differential operator associated with $\ell$ and its inverse at the
end of the section. This equation can be transformed to the form
(\ref{MW}) but all singularities are located on the real axis so the
additional condition of Magnus and Winkler is not satisfied.

Let $f \in \L^2 (-\pi,\pi)$ and $\ep>0$. Denote $\I_+ = (0,\pi)$,
$\I_- = (-\pi,0)$. Consider the differential equation
\begin{gather}\label{diffeq} \ell [\h](x)= f (x) \quad
\text{a.e. on} \quad  (-\pi,\pi)
\end{gather}
assuming that
\begin{gather} \label{as AC h}
\text{the functions} \quad \h \quad \text{and} \quad \ep \sin
(\theta) \h' + \h \quad \text{are in} \quad AC_{\loc} (\I_- \cup
\I_+),
\end{gather}
i.e., are absolutely continuous on each closed subinterval of $\I_-
\cup \I_+$.

\begin{lemma}
Let $\h$ satisfy \eqref{as AC h}. Then $\h$ is a solution of the
equation $\ell [\h] (x) = f (x)$ if and only if $\h$ has the form
\begin{gather} \label{e h1}
\h (\theta) = |\cot (\theta/2)|^{1/\ep}
\left( k_2^\pm - \int_{\pm \pi/2}^\theta f(t) |\tan (t/2)|^{1/\ep}
dt \right) +\int_0^\theta f(t) dt + k_1^\pm, \quad \theta \in
\I_\pm,
\end{gather}
where  $k_1^\pm$ and $k_2^\pm$ are arbitrary constants.
\end{lemma}

The proof is based on direct calculations.

\begin{proposition}
Assume that $\ep \in (0,2)$. A function $\h\in \L^2 (-\pi,\pi)$
satisfies \eqref{as AC h} and is a solution of the equation
$\ell[\h]=f$ with $f \in \L^2 (-\pi,\pi)$ if and only if $\h$ has
the form
\begin{gather} \label{e h}
\h (\theta) = - \left| \cot \frac \theta 2 \right|^{1/\ep}
\int_0^\theta f(t) \left| \tan \frac t 2 \right|^{1/\ep} dt +
\int_0^\theta f(t) dt + k_1^\pm, \quad \theta \in \I_\pm,
\end{gather}
where $k_1^\pm$ are arbitrary constants.
\end{proposition}

\begin{proof}
Assume that $\h$ is an $\L^2 (-\pi,\pi)$-solution of $\ell[\h]=f$.
Then it has the form \eqref{e h1}. Note that
\begin{gather*} 
f (t) \left| \tan \frac t 2 \right|^{1/\ep} \in \L^1 (0,\delta)
\quad \text{for any} \quad \delta \in (0,\pi).
\end{gather*}
 Therefore there exist the finite limit
\[
C_1 := \lim_{\theta \to +0} \left( k_2^\pm - \int_{\pm \pi/2}^\theta
f(t) \left| \tan \frac t 2 \right|^{1/\ep} dt \right).
\]
If $C_1 \neq 0$, then \eqref{e h1} implies that $|\h (\theta)| \geq
|C_2| \ \theta^{-1/\ep} $ for $\theta>0$ small enough, where $C_2
>0$. Since $\ep <2$, we see that $\h \not \in \L^2 (0,\pi)$. This
shows that $C_1 = 0$, and therefore $\h$ has the form \eqref{e h} on
$(0, \pi)$. Similarly, one can show that $\h$ has the form \eqref{e
h} on $(-\pi,0)$.

Let us prove that any function $\h$ of the form \eqref{e h} belongs
to $\L^2 (0,\pi)$ (the proof for $\L^2 (-\pi,0)$ is the same). It is
enough to check that $\h \in \L^2 (0,\delta)$ and $\h \in \L^2
(\pi-\delta, \pi)$ for sufficiently small $\delta>0$.

For $\theta \in (0,\delta)$ 
, we have
\begin{gather*} 
\int_0^\theta | f(t) | \ \left| \tan \frac t 2 \right|^{1/\ep} dt
\leq C_3 \| f \|_{\L^2} \left( \int_0^\theta t^{2/\ep} dt
\right)^{1/2} = C_3  \| f \|_{\L^2} \ \theta^{1/2+ 1/\ep}.
\end{gather*}
Hence,
\begin{gather} \label{e h at 0}
\left| \cot \frac \theta 2 \right|^{1/\ep}  \int_0^\theta  |f(t) | \
\left| \tan \frac t 2 \right|^{1/\ep} dt \leq 2^{-1/\ep} C_3 \ \| f
\|_{\L^2} \ \theta^{1/2},
\end{gather}
and we finally see that $\h \in \L^2 (0,\delta)$.

For $\theta \in (\pi-\delta,\pi)$, 
we have
\begin{gather*} 
\int_0^\theta | f(t) | \ \left| \tan \frac t 2 \right|^{1/\ep}dt
\leq C_4 + 2^{1/\ep} \, \| f \|_{\L^2} \, \left(
\int_{\pi-\delta}^\theta (\pi-t)^{-2/\ep} dt \right)^{1/2} \leq C_5
+ 2^{1/\ep}
 \, \| f \|_{\L^2} \, (\pi-\theta)^{1/2-1/\ep} .
\end{gather*}
Hence,
\[
\left| \cot \frac \theta 2 \right|^{1/\ep} \int_0^\theta  | f(t) | \
\left| \tan \frac t 2 \right|^{1/\ep}  dt \leq C_6 (\pi -
\theta)^{1/\ep} + C_7 \, \| f \|_{\L^2} (\pi - \theta)^{1/2}.
\]
So $\h \in \L^2 (\pi-\delta, \pi)$.
\end{proof}

In particular, we have proved that
\[
\lim_{\theta \to \pm0} h(\theta) = k_1^\pm \quad \text{and} \quad
\lim_{\theta \to \pm\pi\mp0} h(\theta)= k_1^\pm + \int_0^{\pm\pi}
f(t) dt
\]
hold for any $\L^2$-solution $\h$. This implies that the condition
\begin{gather} \label{as hper}
\h \quad \text{is continuous on} \quad [-\pi,\pi] \quad \text{and
periodic}
\end{gather}
is fulfilled exactly when
\begin{gather*} 
k_1^+ =  k_1^- \quad \text{and} \quad f \perp \one .
\end{gather*}

Let the symbol $\W_{2p}^k(-\pi,\pi)$ stand for the subspace of the
space $\W_2^k(-\pi,\pi)$ consisting of periodic functions, i.e.,
functions satisfying  the conditions $u^{(i)}(\pi)=u^{(i)}(-\pi)$
$(i=0,1,\ldots,k-1)$. The norm in this space coincides with that of
the Sobolev space $\W_{2}^k(-\pi,\pi)$.

\begin{proposition}
\label{explicitform} Let $\ep \in (0,2)$, $f \in \L^2 (-\pi,\pi)$,
and $\int_{-\pi}^{\pi} f(\theta) d\theta =0$. Then an
$\L^2$-solution of $\ell [h] = f$ satisfies \eqref{as hper} if and
only if
\begin{gather} \label{e h cont}
\h (\theta) = - \left| \cot \frac \theta 2 \right|^{1/\ep}
\int_0^\theta f(t) \left| \tan \frac t 2 \right|^{1/\ep} dt +
\int_0^\theta f(t) dt + k_1 \quad \text{for a.a.} \quad \theta \in
(-\pi,\pi),
\end{gather}
where $k_1$ is an arbitrary constant. Moreover, any function $\h$ of
the form  \eqref{e h cont} possesses the following properties:
\begin{description}
\item[(i)] $\h \in AC [-\pi,\pi]$,  $\h' \in \L^2 (-\pi,\pi)$,
and
\begin{gather} \label{e |h'|}
\| \h' \|_{\L^2} \leq K (|k_1| + \| f \|_{\L^2}) ,
\end{gather}
where $K$ is a constant independent of $f$.
\item[(ii)] $\sin (\theta) \h' \in AC [-\pi,\pi]$ and
$(\sin (\theta) \h')' \in \L^2 (-\pi,\pi)$.
\end{description}
\end{proposition}

\begin{proof}
To show that $\h' \in \L^2 (-\pi,\pi)$, it is enough to prove that
$\h' \in \L^2 (0,\delta)$ for any $\delta>0$ small enough. Since
\[
\h' (\theta) = - \frac 1{\ep \sin \theta} \cot^{1/\ep} \frac \theta
2 \int_0^\theta f(t) \tan^{1/\ep} \frac t 2 dt, \quad \theta \in
(0,\delta),
\]
it is sufficient to show that
\begin{gather} \label{e f Hardy}
\int_0^\delta \theta^{-2- 2/\ep} \left( \int_0^\theta f(t) t^{1/\ep}
dt \right)^2 d\theta \leq C_1 \int_0^\delta |f(\theta)|^2 d\theta,
\end{gather}
for any $f \in \L^2 (0,\delta)$.

Denote $g(t) := f(t) t^{1/\ep}$. Then \eqref{e f Hardy} takes the
form
\begin{gather} \label{e g Hardy}
\int_0^\delta \theta^{-2- 2/\ep} \left( \int_0^\theta g(t) dt
\right)^2 d\theta \leq C_1 \int_0^\delta |g(\theta)|^2
\theta^{-2/\ep} d\theta \ .
\end{gather}
This is a weighted norm inequality for the Hardy operator. Applying
\cite{M72} (see also \cite{S84} and references therein), we see that
\begin{gather}
\sup_{\theta \in [0,\delta)} \left( \int_\theta^\delta t^{-2- 2/\ep}
dt \right)^{1/2} \left( \int_0^\theta
\left(\theta^{-2/\ep}\right)^{1-2} \right)^{1/2} < \infty,
\end{gather}
and therefore \eqref{e g Hardy} holds true. It is easy to see that
the latter implies \eqref{e |h'|} and statement (i) of the theorem.

If $f \perp \one$, then $\sin (\theta) \h' + \h \in \W_{2p}^1
(-\pi,\pi)$ and, by statement (i), so is $\sin (\theta) \h'$.
\end{proof}

Introduce the space $X_2=\{h\in \W_{2p}^1(-\pi,\pi):\ (\sin \theta)
h_\theta\in \W_2^1(-\pi,\pi)\}$ endowed with the norm
$$\|h\|_2^2=\|h_\theta\|_{\L^2(-\pi,\pi)}^2+\|(\sin \theta)
h_\theta\|_{\W_2^1(-\pi,\pi)}^2.$$

Denote by  $X_2^0$ the subspace of $X_2$ comprising the functions
$h$ with the property $\int_{-\pi}^{\pi}h(\theta)\,d\theta=0$. As a
consequence of the definitions, we obtain that if $h \in X_2$, then
the function $(\sin \theta) \, h_\theta$ is absolutely continuous (may be
after a change on a set of zero measure) and
$$
(\sin \theta) \, h_\theta|_{\theta=0}=(\sin \theta) \,
h_\theta|_{\theta=\pi}=(\sin \theta) \, h_\theta|_{\theta=-\pi}=0.
$$
Let the symbol $\L_p^2(-\pi,\pi)$ stand for the subspace of
$\L^2(-\pi,\pi)$ comprising the functions $f$ with the property
$\int_{-\pi}^{\pi}f(\theta)\,d\theta=0$.

We write below a set of corollaries of Proposition
\ref{explicitform}. We also give the alternative prove of this
result using the Galerkin method in Appendix \ref{a A}.

Denote by $L$ the operator acting in $\L^2(-\pi,\pi)$ and defined by
\begin{equation} \label{e L}
L f = \ell[f] \quad \text{for} \quad f \in \dom (L) :=
X_2 .
\end{equation}
Clearly, $\ker (L) = \{c\one , \ c \in \C\} $.

Let us denote by $\wt L$ the restriction of $L$ on $\L_p^2(-\pi,\pi)
\left(= (\ker (L) )^{\perp} \right) $,
\begin{equation} \label{e wtL}
\wt L := L \uph \L_p^2(-\pi,\pi), \quad \dom(\wt L) := \dom (L) \cap
\L_p^2(-\pi,\pi).
\end{equation}
It follows from the remark after (\ref{as hper}) that $\ran L
\subset \L_p^2(-\pi,\pi)$. So $\wt L$ is an operator in the Hilbert
space $\L_p^2(-\pi,\pi)$.

To find the inverse operator $\wt L^{-1}$, let us symmetrize (\ref{e h cont}) as
\begin{equation} \label{e h cont2}
h (\theta) = - \left| \cot \frac \theta 2 \right|^{1/\varepsilon}
\int_0^\theta f(t) \left( \left| \tan \frac t 2
\right|^{1/\varepsilon} - \left| \tan \frac \theta 2
\right|^{1/\varepsilon}\right) dt + k_1 \quad \text{for a.a.} \quad
\theta \in (-\pi,\pi).
\end{equation}
Solving the equation $(h,1) = 0$, we get
\begin{equation}
\label{const k1}
 k_1 =\frac{1}{2\pi}
\int\limits_{-\pi}^{\pi}\left[ \left| \cot \frac \theta 2
\right|^{1/\varepsilon} \int_0^\theta f(t) \left( \left| \tan \frac
t 2 \right|^{1/\varepsilon} - \left| \tan \frac \theta 2
\right|^{1/\varepsilon}\right) dt \right] d\theta
\end{equation}
So for $f \in \ran(L)$, we have $\wt L^{-1} f = h$ with $h$  defined by
(\ref{e h cont2})-(\ref{const k1}).

\begin{corollary} \label{closed}
\item  $(i)$  The operator $L$ defined by (\ref{e L}) is a closed operator in $\L^2(-\pi,\pi)$ (with the dense domain $X_2$).
\item $(ii)$  Its kernel $\ker (L)$ is the one-dimensional subspace of constants $\{c\one , \ c \in \C\} $.
\item $(iii)$ The range $\ran (L)$ of $L$ is the orthogonal complement to $\ker (L)$, $\ran (L) = \L_p^2(-\pi,\pi)$.
\item $(iv)$  The operator $\wt L$ defined by (\ref{e wtL}) (and
acting in $\L_p^2(-\pi,\pi)$) has a compact inverse $\wt L^{-1}$.
\item $(v)$ The operator $\wt L: X_2^0 \to \L_p^2(-\pi, \pi)$ is an isomorphism of $X_2^0$ onto $\L_p^2(-\pi, \pi)$.
\end{corollary}

\begin{proof}
(i)-(iii) Let $k_1=k_1(f)$ be the linear functional defined by (\ref{const k1}).
It is easy to see that
\begin{equation} \label{e k1 bound}
k_1 \quad \text{is bounded on} \quad \L^2 (-\pi,\pi).
\end{equation}
This and Proposition \ref{explicitform} imply immediately that
$\ran(L) = \L_p^2(-\pi,\pi)$. So $\dom (\wt L) = \L_p^2(-\pi,\pi)$.
It follows from (\ref{e |h'|}) and
\begin{equation*} 
\| h' \|_{\L^2 (-\pi,\pi)} \geq \| h \|_{\L^2 (-\pi,\pi)}, \quad h
\in \L_p^2 (-\pi,\pi),
\end{equation*}
that $\wt L^{-1}$ is a bounded linear operator on
$\L_p^2(-\pi,\pi)$. Therefore $\wt L$ is closed and so is $L$.

(iv)-(v) It follows from (\ref{e |h'|}) and (\ref{e k1 bound}) that
$\wt L^{-1}$ is bounded as an operator from $\L_p^2(-\pi, \pi)$ onto
$X_2^0$. This proves statements (iv) and (v).
\end{proof}


\begin{remark} \label{r adj}
The  adjoint operation $\ell^*$ is given by
\begin{equation*}
\ell^* [\h]:= \ep \frac {d}{d \theta} \left(
\sin (\theta) \frac {d\h }{d\theta}\right) - \frac{d \h}{d \theta} ,
\quad \theta \in (-\pi,\pi).
\end{equation*}
It is easy to see that the adjoint operator $L^*$ is unitary equivalent to $L$ and has the form
\begin{equation*}
L^* \h = \ell^* [\h], \qquad \dom (L^*) = X_2.
\end{equation*}
Indeed, $L = JL^{*}J$, where $(Jh)(\theta) = h( \pi - \theta)$.
In turn, this implies that the statements analogous to that of
Corollary \ref{closed} are valid for $L^*$.
Note that $J$ is a signature operator, i.e., $J=J^*=J^{-1}$,
so the operator $L$ is $J$-self-adjoint and
belongs to the class of the $PT$-symmetric operations.
\end{remark}


\section{The ill-posedness of the Cauchy problem for the forward-backward heat equation.}

The linearized model of the thin film dynamic (\ref{heat PDE}) was
derived without taking into account the smoothing effect of the
surface tension. It's very natural to expect that a drop of fluid
will detach itself from the "ceiling" of the rotating cylinder and
it will inevitably fall down that perfectly fits into the ill-posed
nature of the Cauchy problem for the forward-backward heat equation
(\ref{heat PDE}). The intuition based on the classic theory of
backward heat equation, says that global in time classic solutions
can exist only for some class of analytic in vertical strip $(-\pi,
\pi)$ functions with exponentially fast decaying Fourier
coefficients.

From there on in this section we assume that the parameter $0 <
\varepsilon < 2 $.  Let us consider the parabolic problem
\begin{gather}
\label{eqno6}
 h_t+L h=0,\  \ \ \theta\in
(-\pi,\pi), \quad t\in (0,T), \\
\label{eqno7} \ h|_{t=0}=h_0(\theta),\ \ u(-\pi,t)=u(\pi,t).
\end{gather}
Note that, after the change of variables $\theta\to -\theta$,
equation (\ref{eqno6}) can be  replaced with the equation
$$
u_t-Lu=0.
$$

Let $Q=(-\pi,\pi)\times (0,T)$. We prove in this section that the
problem (\ref{eqno6}), (\ref{eqno7}) is ill-posed in the classes of
finite smoothness. In what follows, the symbol $(\cdot,\cdot)$
stands for the inner product in the space $\L^2$ in the
corresponding domain ($Q,Q_1,\ldots$).

\begin{definition} \label{d gs1} By a generalized solution to the problem (\ref{eqno6}),
(\ref{eqno7}) from the space $\W_{2p}^{1,0}(Q)$ we imply a function
$h \in \L^2(0,T;\W_{2p}^1(-\pi,\pi))\cap C([0,T];\L^2(-\pi,\pi))$
such that
$$
-(h,v_t)-(\sin
\theta\,h_\theta,v_\theta)+(h_\theta,v)=\int_{-\pi}^\pi
h_0(\theta)v(\theta,0)\,d\theta
$$
for all $ v\in \L^2(0,T;\W_{2p}^1(-\pi,\pi)):$ $v_t\in \L^2(Q),
v(\theta,T)=0$.
\end{definition}

Note that we can use other definitions of generalized solutions (see
\cite{LadSolUr}). In particular, it is possible to prove that the
above definition is equivalent to the following definitions.

\begin{definition} \label{d gs2} A function $h\in \L^2(0,T;\W_{2p}^1(-\pi,\pi))\cap
C([0,T];\L^2(-\pi,\pi))$ such that $h|_{t=0}=u_0(\theta)$ is called
a generalized solution to the problem (\ref{eqno6}), (\ref{eqno7})
if
$$
-(h,v_t)-(\sin \theta\,h_\theta,v_\theta)+(h_\theta,v)=0
$$
for all $ v\in \L^2(0,T;\W_{2p}^1(-\pi,\pi)):$ $v_t\in \L^2(Q),
v(\theta,T)=0, v(\theta,0)=0$.
\end{definition}

\begin{definition} \label{d gs3} A function $h\in \L^2(0,T;\W_{2p}^1(-\pi,\pi))\cap
C([0,T];\L^2(-\pi,\pi))$ such that $h|_{t=0}=h_0(\theta)$ and
$h_t\in \L^2(0,T;\W_{2p}^{-1}(-\pi,\pi))$ is called a generalized
solution to the problem (\ref{eqno6}), (\ref{eqno7}) if
$$
h_t + Lh=0
$$
and this equality holds in the space
$\L^2(0,T;\W_{2p}^{-1}(-\pi,\pi))$ {\rm (}$\W_{2p}^{-1}(-\pi,\pi)$
is the negative space constructed on the pair
$\W_{2p}^{1}(-\pi,\pi)$, $\L^2(-\pi,\pi)${\rm )}.
\end{definition}

\begin{theorem}[\bf nonexistence]
\label{nonexist} Assume that there exists $k\in \nn$ such that
$h_0\in \W_{2p}^k(-\pi,\pi)$, $\supp h_0\in (0,\pi)$, and $h_0\notin
\W_{2p}^{k+1}(-\pi,\pi)$. Then there is no a generalized solution
from the space $\W_{2p}^{1,0}(Q)$ for  problem (\ref{eqno6}),
(\ref{eqno7}).
\end{theorem}

\begin{proof}
Assume the contrary, i.e., that such a solution exists. Let $\supp
h_0\subset (\delta,\pi-\delta)$. Put $Q_0=(\delta,\pi-\delta)\times
(0,T/2)$. We have
\begin{equation}
\label{eqno5.1} -(h,v_t)-(\sin
\theta\,h_\theta,v_\theta)+(h_\theta,v)=\int_{-\pi}^\pi
h_0(\theta)v(\theta,0)\,d\theta
\end{equation}
for all functions $v$ from Definition \ref{d gs1} such that $\supp
v\subset Q_0\cup \{(\theta,0):\ \theta\in
(\delta,\pi-\delta)\}$. Make the change of variables $\tau=T/2-t$.
Then the function $h$ is a generalized solution of the equation
$$
h_\tau-Lh=0
$$
in $Q_0$ and $h$ satisfies the condition
\begin{equation}
\label{eqno5.2}
 h|_{\tau=T/2}=h_0(\theta).
\end{equation}
Here, we understand a generalized solution in the sense that the
function $h \in \L^2(0,T/2;\W_{2}^1(\delta,\pi-\delta))\cap
C([0,T/2];\L^2(\delta,\pi-\delta))$ is such that
\begin{equation}
\label{eqno5.3} -(h,v_\tau)+(\sin
\theta\,h_\theta,v_\theta)-(h_\theta,v)=-\int_{\delta}^{\pi-\delta}
h_0(\theta)v(\theta,T/2)\,d\theta
\end{equation}
for all $ v\in \L^2(0,T/2;\W_{2}^1(\delta,\pi-\delta)):$ $v_t\in
\L^2(Q_0), v(\theta,0)=0$, $v|_{\theta=\delta}=
v|_{\theta=\pi-\delta}=0$. Now construct a function $h_1$ being a
solution to the problem
$$
h_\tau-Lh=0,\ \ h|_{\theta=\delta}= h|_{\theta=\pi-\delta}=0, \
h_{\tau=T/2}=h_0(\theta),
$$
in the domain $Q_1=(\delta,\pi-\delta)\times (T/2,T)$. Since $k\geq
1$, we see that a solution of this problem exists and belongs at
least to the space $\W_2^{2,1}(Q_1)$ (see  \cite[Theorem
3.6.1]{LadSolUr}). This solution satisfies the integral identity
\begin{equation}
\label{eqno5.4} -(h_1,v_\tau)+(\sin
\theta\,h_{1\theta},v_\theta)-(h_{1\theta},v)=\int_{\delta}^{\pi-\delta}
h_0(\theta)v(\theta,T/2)\,d\theta
\end{equation}
for all $v\in \L^2(T/2,T;\W_{2}^1(\delta,\pi-\delta)):$ $v_\tau\in
\L^2(Q_1), v(\theta,T)=0$, $v|_{\theta=\delta}=
v|_{\theta=\pi-\delta}=0$. Put $Q_2=(\delta,\pi-\delta)\times
(0,T)$. Summing (\ref{eqno5.3}) and (\ref{eqno5.4}) we conclude that
\begin{equation}
\label{eqno5.5} -(h_2,v_\tau)+(\sin
\theta\,h_{2\theta},v_\theta)-(h_{2\theta},v)=0
\end{equation}
for all $v\in \L^2(0,T;\W_{2}^1(\delta,\pi-\delta)):$ $v_\tau\in
\L^2(Q_2), v(\theta,T)=0$, $v|_{\theta=\delta}=
v|_{\theta=\pi-\delta}=0$. Here the function $h_2$ coincides with
$h$ for $t\in (0,T/2)$ and with $h_1$ for $t\in (T/2,T)$. So the
function $h_2$ is a generalized solution of the equation
$$
h_\tau-Lh=0
$$
in $Q_2$ in the sense of the integral identity (\ref{eqno5.5}). By
\cite[Theorem III.12.1]{LadSolUr}, we have $u\in C^{\infty}(Q_2)$.
Therefore, $h_2(\theta,T/2)=h_1(\theta,T/2)=h_0(\theta)\in
C^{\infty}(\delta,\pi-\delta)$. This contradicts to the fact that
$h_0\notin \W_{2p}^{k+1}(-\delta,\delta)$.
\end{proof}

\begin{remark} It is easily seen from the proof that it does not matter where the
support of $h_0$ lies. The main condition is that $h_0\notin
\W_2^{k+1}(\delta,\pi-\delta)$ for some $\delta>0$. So the theorem
 can be strengthened.
\end{remark}

\begin{theorem}[instability]
\label{instab} Assume that the problem (\ref{eqno6}), (\ref{eqno7})
is densely solvable in the following sense: for $h_0$ from some set
$K$ of smooth functions that is dense in the space
$\W_{2p}^k(-\pi,\pi)$ $(k\geq 2)$,  problem (\ref{eqno6}),
(\ref{eqno7}) has a generalized solution $h$ in the sense of
Definition \ref{d gs1}. Then there is no constant $c>0$ such that,
for every generalized solution of the problem (\ref{eqno6}),
(\ref{eqno7}) with initial value $h_0\in K$, the estimate
\begin{equation}
\label{eqno8} \|h\|_{\L^2(Q)}\leq c \|h_0\|_{\W_{2p}^k(-\pi,\pi)}
\end{equation}
holds.
\end{theorem}

\begin{proof}

We use the arguments of Theorem \ref{nonexist}. Let us find a
function $h_0\in \W_{2p}^k(-\pi,\pi)$ such that  $\supp h_0\in
(0,\pi)$,  and $h_0\notin \W_{2p}^{k+1}(-\pi,\pi)$. Find $\delta>0$
such that $\supp u_0\subset (\delta,\pi-\delta)$. Put
$Q_0=(\delta/2,\pi-\delta/2)\times (0,T/2)$.

Assume the contrary, i.e., that (\ref{eqno8}) holds. Construct a
sequence $h_{0n}\in K:$ $\|h_{0n}-h_0\|_{\W_{2}^k(-\pi,\pi)}\to 0$
as $n\to \infty$. In this case,
\begin{equation}
\label{eqno9} \|h_{0n}\|_{\W_{2}^{k+1}(3\delta/4,\pi-3\delta/4)}\to
\infty\ \ \textrm {as}\ \ n\to \infty.
\end{equation}
Denote by $h_n$ the corresponding generalized solutions to our
parabolic problem. As in the proof of Theorem \ref{nonexist}, the
change of variables  $\tau=T/2-t$, shows that the functions
$\tilde{h}_n=h_n(\theta,T/2-\tau)$ satisfy the integral identity
\begin{equation}
 -(\tilde{h}_n,v_\tau)+(\sin
\theta\,\tilde{h}_{n\theta},v_\theta)-(\tilde{h}_{n\theta},v)=-\int_{\delta/2}^{\pi-\delta/2}
h_{0n}(\theta)v(\theta,T/2)\,d\theta
\end{equation}
for all $ v\in \L^2(0,T/2;\W_{2}^1(\delta/2,\pi-\delta/2)):$ $v_t\in
\L^2(Q_0), v(\theta,0)=0$, $v|_{\theta=\delta/2}=
v|_{\theta=\pi-\delta/2}=0$.

Thus, the functions  $\tilde{h}_n$ are generalized solutions to a
parabolic equation in $Q_0$.  Using Theorems III.8.1 and III.12.1
and Theorem IV.10.1 in \cite{LadSolUr}, we obtain that the function
$h_n(\theta,T/2-\tau)$ is infinitely differentiable in $Q_0$ and the
norm of this function in any H\"{o}der space
$H^{2+\alpha,1+\alpha/2}(\overline{Q_1})$ with
$Q_1=(\delta',\pi-\delta')\times (\varepsilon_0,T/2)$
($\delta'>0,\varepsilon_0>0$) is estimated by some constant
depending on $\delta',\varepsilon_0,\alpha $, and the norm
$\|h_n\|_{\L^2(Q_0)}$. In view of \ref{eqno8} with $h_n,h_{0n}$
substituted for $h,h_0$ and the fact that the norms
$\|h_{0n}\|_{\W_{2}^k(-\pi,\pi)}$ are bounded, we can assume that
this constant is independent of $n$. As a consequence, we have the
estimate
\begin{equation}
\label{eqno10}
\|h_n(\theta,0)\|_{\W_{2}^{k+1}(3\delta/4,\pi-3\delta/4)}\leq c
\|h_n\|_{\L^2(Q_0)}
\end{equation}
where the constant $c$ is independent of $n$. Comparing
\ref{eqno10}, \ref{eqno9}, we arrive at a contradiction.
\end{proof}

\section{The completeness property for the operator $L$
\label{s sp prop}}

As in the previous section we restrict the parameter to the interval
$0<\ep <2$. In this section, we prove that the system of all
eigenvectors and generalized eigenvectors of the operator $L$ is
complete in $\L^2 (-\pi, \pi)$. In particular, this implies that $L$
has infinite number of eigenvalues.

Denote by $\H^{2,0} (\DD) $ and $ \H^{2,0} (\bar{\C} \setminus \DD)
$ the subspaces of the Hardy spaces $\H^2 (\DD)$ and $\H^2
(\overline{\C} \setminus \DD)$ , respectively (see e.g.
\cite[Section 2.1]{G81}) that are orthogonal to the function $\one
(\equiv 1)$. In the sequel, we use the standard identification of
the function $u(z) \in H^2 (\DD)$ with the function $u (e^{i\theta})
:= \lim\limits_{r \to 1-0} u(r e^{i\theta})$, which belongs to $\L^2
(-\pi,\pi)$, and also use the similar agreement for $u(z) \in H^2
(\bar{\C} \setminus \DD) $. Then $\H^2 (\DD)$ and $ \H^2 (\bar{\C}
\setminus \DD) $ are the subspaces of $\L^2 (-\pi,\pi)$ and $\H^2
(\DD) \cap \H^2 (\bar{\C} \setminus \DD) = \{c\one , \ c \in \C\}$.
In these terms, the space $\L^2 (-\pi,\pi)$ admits the orthogonal
decomposition
\begin{equation}  \label{L2 dec}
\L^2 (-\pi,\pi) = \H^{2,0} (\DD) \oplus \{c\one , \ c \in \C\}
\oplus \H^{2,0} (\bar{\C} \setminus \DD).
\end{equation}

Define the operator $L_{\fin}$ in the Hilbert space $\L^2
(-\pi,\pi)$ by $L_{\fin} \h := \ell [\h]$, $ \dom(L_{\fin}) =
P_{\fin}$, where $\ell$ is the differential expression defined in
Section \ref{s DE} and $P_{\fin}$ is the set of finite trigonometric
polynomials
\[
\h (\theta) = (2 \pi)^{-1/2} \sum_{n=-N}^N v_n e^{in \theta}, \quad
N<\infty , \quad  v_n \in \C.
\]
It is easy to see that $L_{\fin}^*$ is densely defined, and hence
the closure
\begin{equation*}
L_{\min}:=\overline{L_{\fin}}
\end{equation*}
exists as an operator in $\L^2 (-\pi,\pi)$.

Let $L$ be the indefinite convection-diffusion operator defined by
(\ref{e L}) and let $\wt L$ be its restriction defined by (\ref{e
wtL}). It is easy to see that Remark \ref{r adj} implies that
$L_{\fin}^* = L^*$ and therefore $L_{\min} = L$. Below we give
another proof of this fact using the results of \cite{D07}. This
proof allows us to use the orthogonal decomposition of $L$ obtained
in \cite{D07} (see also \cite{ChugPel}).

\begin{proposition}[Theorems 11 and 13 in \cite{D07}] \label{Lpm HS}
\label{ortogdecom} \item \it(i) The operator $L_{\min}$ admits the
orthogonal decomposition $L_{\min} = L_- \oplus \mathbf{0} \oplus
L_+$ with respect to (\ref{L2 dec}).
\item \it{(ii)} The operators $L_\pm$ are invertible, and their inverses
 $L_\pm^{-1}$ are Hilbert-Schmidt operators.
\item \it{(iii)} $L_+^{-1}$ and $(-L_-)^{-1}$ are unitary equivalent.
\end{proposition}

\begin{proposition} \label{p L=L}
If $\ep \leq 2$, then $L=L_{min}$ and $\wt L = L_- \oplus L_+$.
\end{proposition}

\begin{proof}
By Corollary \ref{closed} (i), $L$ is a closed extension of $L_{\fin}$. Hence
$L_{\min} \subset  L $. Let us show that
$L=L_{\min}$.

By Proposition \ref{Lpm HS}, the operator $L_+^{-1} \oplus L_-^{-1}$
is compact and is acting in $\L_p^2 (-\pi,\pi)$.
So $\sigma (L_- \oplus L_+)$ is at most countable. Then
\begin{equation} \label{e ranLmin}
\ran (L_- \oplus L_+ - \la I) = \L_p^2 (-\pi,\pi) \quad  \text{for
any } \quad \la \in \rho(L_- \oplus L_+).
\end{equation}
By Corollary \ref{closed} (iv), the operator $\wt L^{-1} $ acting in
$\L_p^2 (-\pi,\pi)$ is compact. So $\wt L$ possesses the same
properties, that is, $\sigma(\wt L)$ is at most countable and
(\ref{e ranLmin}) holds for $\wt L$.

Assume that $L_+ \oplus L_- \subsetneqq \wt L$ (which is equivalent
to $L_{\min} \subsetneqq L$). Then Eqs.~(\ref{e ranLmin}) for $L_+
\oplus L_-$ and $\wt L$ imply that $\la \in \sigma_p (\wt L)$
whenever $\la \in \rho(\wt L) \cap \rho (L_+ \oplus L_-)$, a
contradiction.
\end{proof}

\begin{definition}[e.g. \cite{GohKrein}]
By $\Sclass_p$, $0<p < \infty$, we denote the class of all bounded
linear operators $A$ acting on a Hilbert space $H$ for which
$$  |A|_p := \left( \sum\limits_{j = 1}^{\infty} s_j^p(A) \right)^{1/p} < \infty$$
where $s_j(A)$ are singular numbers of $A$, i.e, eigenvalues of the
self-adjoint operator $(A^*A)^{1/2}$ that are enumerated in
decreasing order, counted with multiplicities.
\end{definition}

Two $\Sclass$-classes were given special names:  $\Sclass_2$ is the
class of Hilbert-Schmidt operators and $\Sclass_1$ is the class of
nuclear operators. It was proved by Davies \cite{D07} that the
operators $A_\pm^{-1}$, where
\begin{equation} \label{e Apm}
A_\pm := -i L_\pm,
\end{equation}
belong to the class $\Sclass_2$ and so does $\wt L^{-1}$.
We prove in this section that, actually, the operator $\wt L^{-1}$  is nuclear.

We need the following result (see \cite{GohMar}, its weaker version
can be found e.g. in \cite[Section III.7.8]{GohKrein}):

\begin{theorem}[Gohberg, Markus]
\label{sigma}
If $0 < p \leq 2$, then the linear operator
$A$ acting in a Hilbert space $H$ belongs to $\Sclass_p$ if and
only if for at least one orthonormal basis $\{e_j\}$ of
$H$ the inequality
\begin{equation} \label{e S_p ine1}
 \sum\limits_{j=1}^{\infty} |A e_j|^p < \infty
 \end{equation}
holds. In addition,
$$ |A|_p^p \leq \sum\limits_{j=1}^{\infty} |A e_j|^p \leq \sum\limits_{j,k = 1}^{\infty} |(A e_j, e_k)|^p.$$
\end{theorem}

It follows from \cite[Theorem 11 and Eq.~(15)]{D07} that the
operator $ iL_+^{-1} (= A_+^{-1} )$ in the Fourier basis $\{ e_n
\}_1^{\infty}$, $e_n (\theta) = e^{i n \theta}$, is represented by
matrix $(\matA_{m,n})$ which has the following properties
\begin{equation}
\label{est coef four}
\begin{array}{l}
    |\matA_{m,n}| \leq C_1 \, m^{-1 + 1/\varepsilon} \ n^{-1-1/\varepsilon}, \quad m \leq n , \\
    |\matA_{m,n}| \leq C_1 \, m^{-1 - 1/\varepsilon} \ n^{-1+1/\varepsilon}, \quad n < m. \\
\end{array}%
\end{equation}


\begin{proposition}
\label{sigma-class} The operator $\wt L^{-1}$ belongs to the class
$\Sclass_1$  for any  $\ep \in (0,2)$. More precisely:
\item \it{(i)} \, if \, $\ep \in (0,1]$, then $\wt L^{-1} \in \Sclass_p$  for any  $p
>2/3$,
\item \it{(ii)}\, if  \, $\ep \in (1,2)$, then $\wt L^{-1} \in \Sclass_p$  for any
$p > 2 \ep/(\ep + 2)$.
\end{proposition}

\begin{proof}

$ $\\
\textbf{\it{(i)}} \, $\ep \in (0,1]$.

Using (\ref{est coef four}), one can obtain that
\begin{eqnarray}
\|L_+^{-1} e_n \|_{\L^2}^2 = \sum\limits_{m=1}^{\infty} |\matA_{m,n}
|^2  \leq C_1^2 n^{ -2-2/\ep}\sum\limits_{m=1}^{n} m^{-2+2/\ep}  +
C_1^2  n^{ -2+2/\ep} \sum\limits_{m=n+1}^{\infty} m^{-2-2/\ep} \notag \\
\leq C_1^2 n^{-3} + C_1^2 n^{ -2+2/\ep} \int_{n}^{\infty}
m^{-2-2/\ep} dm \leq C_2 n^{-3} . \label{e ep<1 est}
\end{eqnarray}
Hence,
\[
\sum\limits_{n=1}^{\infty} \|L_+^{-1} e_n \|_{\L^2}^p \leq C_3
\sum\limits_{n=1}^{\infty} n^{-3p/2} .
\]
So the Gohberg-Markus criterion (Theorem \ref{sigma}) shows that
$L_+^{-1} \in \Sclass_p$ for any $2/3<p\leq 2$. Thus $L_+^{-1}$
belongs to $\Sclass_p$ for any $ p> 2/3$ and so does $\wt L^{-1}$
due to Propositions \ref{Lpm HS} (ii) and \ref{p L=L}. $
$\\\textbf{\it{(ii)}} \, $\ep \in (1,2)$.

Using (\ref{est coef four}), one can obtain that
\[
\|L_+^{-1} e_n \|_{\L^2}^2 \leq C_1^2 n^{
-2-2/\ep}\sum\limits_{m=1}^{n} m^{-2+2/\ep}  + C_1^2  n^{ -2+2/\ep}
\sum\limits_{m=n+1}^{\infty} m^{-2-2/\ep} \leq C_1^2 n^{-1-2/\ep} +
C_2 n^{-3}\leq C_3 n^{-1-2/\ep} .
\]
Therefore, $\wt L^{-1} \in \Sclass_p$  whenever $ p(\ep + 2)/(2
\ep)> 1$.
\end{proof}

Although the weaker result that the operator $\wt L^{-1} \in
\Sclass_{p}$  for $ p>1$ can be obtained directly from the
factorization of the operator $L$ found by Chugunova and Strauss in
\cite{ChugStr}, the fact that the operator $\wt L^{-1}$ actually
belongs to the class of nuclear operators $\Sclass_1$ is crucial, as
you see below, for the proof of completeness.

Following \cite[Section IV.4]{GohKrein},
we will call an operator $T$ acting in a Hilbert space $H$ \emph{dissipative} if
\begin{equation} \label{e dis}
\Im (T f,f) \geq 0 \quad \text{for all} \quad f \in \dom (T).
\end{equation}

\begin{proposition}
\label{dissipative}
The operators $L_+$, $(-L_+)^{-1}$, $\, -L_-$ and $L_-^{-1}$ are dissipative.
\end{proposition}

\begin{proof}
Using the tridiagonal matrix representations of $A_+ (= -iL+)$ with
respect to the Fourier basis $\{ e^{i n \theta} \}_1^{\infty}$ (see
\cite{ChugPel,D07}), we get:
\begin{gather}
\label{matrix-L_+}
A_+ = (a_{n,m})_1^\infty =  \left[ \begin{array}{ccccc} 1 & -\varepsilon & 0 & 0 & \cdots \\
\varepsilon & 2 & -3 \varepsilon & 0 & \cdots \\
0 & 3 \varepsilon & 3 & -6 \varepsilon & \cdots \\
0 & 0 & 6 \varepsilon & 4 & \cdots \\
\vdots & \vdots & \vdots & \vdots & \ddots \end{array} \right] , \\
a_{n,n} = n, \quad a_{n-1,n}= \frac{\ep}2 n (n-1), \quad
a_{n,n+1}=-\frac{\ep}2 n (n+1) , \quad n=1,2,\dots \ . \notag
\end{gather}
This representation implies that
\begin{equation} \label{e im>0}
\Im (L_+ h , h ) = \Re (A_+ h, h) \geq 0
\end{equation}
for all $h \in \H^{2,0} (\DD) \cap P_{\fin}$. Since $L = \overline{L_{\fin}}$, one gets (\ref{e im>0}) for all $h \in \dom (L^+)$, i.e., $L_+$ is dissipative. Substituting $h= (L_+)^{-1} f$ into (\ref{e im>0}), we see that so is $(-L_+)^{-1}$. Proposition
\ref{Lpm HS} (iii) completes the proof.
\end{proof}

\begin{theorem}[Lidskii, see e.g. Theorem V.2.3 \cite{GohKrein}]
\label{lidskii} If the dissipative operator $A$
acting in a Hilbert space $H$ belongs to the class $\Sclass_1$, then
its system of all eigenvectors and generalized eigenvectors is
complete in $H$.
\end{theorem}

Now the main result of this section can be obtained using Propositions \ref{Lpm HS},
\ref{sigma-class}, \ref{dissipative} and Lidskii's theorem.

\begin{theorem} \label{t Lcomp}
The operator $L$ has infinitely many eigenvalues. The system of its
eigenvectors and generalized eigenvectors is complete in $\L^2
(-\pi, \pi)$.
\end{theorem}

\begin{proof}
By Propositions \ref{sigma-class}, \ref{dissipative} and Lidskii's
theorem, the system of all eigenvectors and generalized eigenvectors
of the operators $L_+^{-1}$, $L_-^{-1}$, and $\wt L^{-1}$ are
complete in $\H^{2,0} (\DD)$,  $\H^{2,0} (\bar{\C} \setminus \DD)$,
and $\L_p^2 (-\pi,\pi)$, respectively. Since $\ker (\wt L^{-1}) = \{
0 \}$, all generalized eigenspaces of $\wt L^{-1}$ corresponding to
its eigenvalues $\alpha_n$  are generalized eigenspaces of $\wt L $
corresponding to eigenvalues $\lambda_n = 1/\alpha_n$. Since all
eigenvalues of the compact operator $\wt L^{-1} $ have finite
algebraic multiplicities, we see that  $\wt L^{-1} $ and $\wt L$
have infinitely many eigenvalues. The finite-dimensional spectral
mapping theorem implies that the completeness property holds for
$\wt L $ and, consequently, for $L$.
\end{proof}

\section{Pure imaginary eigenvalues and the Riesz basis property}

In this section, the following result is obtained:  all eigenvalues
of $L$ are pure imaginary (this statement was proved by Weir
\cite{Weir1} under the additional assumption that $1/\varepsilon
\not \in \Z$). We use this fact to prove that eigenvectors of $L$ do
not form a Riesz basis in $\L^2 (-\pi, \pi)$.

Recall that $L_+$ is an operator in the Hilbert space $\H^{2,0}
(\DD) = H^2 (\DD) \ominus \{c\one, \, c\in \C \}$. We identify  the
function $u(z) \in H^2 (\DD)$ with $u(e^{i\theta}) \in \L^2
(-\pi,\pi)$. Note that
\begin{equation*} \label{e u0=0}
u \perp \one \quad \text{is equivalent to} \quad u (0) = 0.
\end{equation*}
So the last equality holds for all $u \in \H^{2,0} (\DD) $.

Let $u(z)$, $z \in \DD$, be an eigenfunction of the operator $A_+
(=-iL_+)$. Consider the restriction $\ux$ of the function $u$ on the
interval $[0,1) \subset \DD$,
\begin{equation} \label{e ux}
 \ux (x)=u(x) \quad \text{for} \quad x \in [0,1).
\end{equation}
The following proposition obtained by Weir \cite{Weir1} shows that
if $u$ is an eigenfunction of the operator $A_+ (=-iL_+)$, then its
restriction $\ux$ is a solution of a  Sturm-Liouville eigenvalue
problem with real-values coefficients.

\begin{proposition}[\cite{Weir1}] \label{p W1}
Assume that $\lambda$ is an eigenvalue of the operator $L_+$,
$u(e^{i\theta})$ is a corresponding eigenvector, and $\ux$ is
its restriction defined by (\ref{e ux}). Then $\b[\ux] (x) = \mu \ux (x)$
 for all $x \in (0,1) (\subset \DD)$, where $\mu = - 2 i \lambda /\ep$ and the
differential expression $\b$ is  defined by
\begin{gather*}
\b[u] = -\frac 1w ( pu')' , \\
p(x)=  (1-x)^{1+1/\ep} (x+1)^{1-1/\ep}, \quad w(x)= x^{-1}
(1-x)^{1/\ep} (x+1)^{-1/\ep}.
\end{gather*}
\end{proposition}

Let $B_{\max}$ be an operator in $ \L^2 ((0,1); w) $ associated with
the differential expression $\b [\cdot]$ and defined on its maximal
domain
\begin{equation*}
B_{\max} \ux =\b [\ux], \quad \dom(B_{\max}) = \{ \ux \in \L^2
((0,1); w) : \ux,\ux' \in AC_{\loc} (0,1), \quad \b [\ux] \in \L^2
((0,1); w)\}.
\end{equation*}

Note that all points of the interval $(0,1)$ are regular for the
differential expression $\b$, but the endpoints $0$ and $1$ are
singular ($1$ is singular since $p^{-1} \notin \L^1 (1/2,1)$).

\begin{proposition} Let $\ep >0$.
\begin{description}
\item[(i)] $\b$ is in the limit-point case at $0$,
\item[(ii)] $\b$ is in the limit-point case at $1$
exactly when $\ep \leq 1$.
\item[(iii)]
$B_{\max}$ is self-adjoint in $ \L^2 ((0,1); w) $ exactly when
$0<\varepsilon \leq 1$.
\end{description}
\end{proposition}

\begin{proof}
(i) Clearly, $\one$ is a solution of $\b[u]=0$ and $\one \not \in
\L^2 ((0,1/2); w)$. Weyl's alternative (see e.g. \cite[Theorem
5.6]{W87}) completes the proof.

(ii) The general solution of $\b[u]=0$ on $(0,1)$ is
\[
u(x) =  k_1 \int_{1/2}^x \frac 1{p(s)} ds + k_2 , \quad k_1, k_2 \in
\C.
\]
Clearly,
\[
u(x) \asymp k_1 \int_{1/2}^x (1-s)^{-1-1/\ep} ds + k_2 \asymp k_1
(1-x)^{-1/\ep} + k_2, \quad x \to 1-0.
\]
Hence all solutions of $\b[u]=0$ belong to $\L^2 ((1/2,1); w)$ if
and only if $\ep >1$.

(iii) follows from (i) and (ii).
\end{proof}

\begin{proposition} \label{p Carl}
If $u(e^{i\theta}) \in \H^{2,0} (\DD)$, then the function $\ux$ defined by
(\ref{e ux}) belongs to $\L^2 ((0,1); w)$.
\end{proposition}

\begin{proof}
By (\ref{e u0=0}), we have $\ux(0)=0$. Since $\ux$ is analytic at
$0$, we see that $\ux \in \L^2 ((0,1/2); w)$. The measure $w(x)dx$
on $ [1/2,1)$ induces a measure $M (S) := \int_{S \cap [\frac 12,1)}
w(x)dx$ on $\DD$. For any sector
\[
S=\{ r e^{i\theta}: 1-l \leq r <1, |\theta - \theta_0|<l \}, \quad l
\in (0,1),
\]
we have
\[
M (S) \leq 2 \int_{\max \{ 1/2,1-l\} } ^1 dx \leq 2 l
\]
since $\max_{x \in [1/2,1)} w(x) \leq 2$. So $M(\cdot)$ is a
Carleson measure (see e.g. \cite[Sec. 4.3]{G81}). Therefore,
\[
\int_{1/2}^1 |\ux (x)|^2 w(x) dx \leq C_1 \| u \|_{\H^{2,0}}^2 \ ,
\]
where $C_1$ is a constant independent of $u$. This completes the
proof.
\end{proof}

\begin{proposition} \label{p real}
Let $\ep \in (0,2)$. Then all eigenvalues of the
operator $A_+(=-iL_+)$ are real and positive.
\end{proposition}

\begin{proof}
Let $u(e^{i\theta}) \in \H^{2,0} (\DD)$ and $L_+ u = \lambda u$. By
Proposition \ref{p Carl}, $ \ux \in \L^2 ((0,1);w)$. Proposition 
 \ref{p W1} implies that $\b[\ux]=\mu \ux $ with $\mu= - 2 i \lambda /\ep$.
Let us split the interval $(0,2)$ into two parts.

If $\ep \leq 1$ then the proof is simple.
 Clearly, $\ux \in \dom(B_{\max})$, and therefore $\mu$ is an eigenvalue of the
nonnegative self-adjoint operator $B_{\max}$. Thus, $\mu \geq 0$.

If $ 1<\ep < 2$ then the proof requires additional analysis. By
Proposition \ref{explicitform}, $g(e^{i\theta}):=\frac {du}{d\theta}
(e^{i\theta})\in \L^2 (-\pi,\pi) $. It is easy to see from the
representation $u(e^{i\theta}) = (2\pi)^{-1/2} \sum_{n=1}^{\infty}
v_n e^{in\theta}$ and that $g(e^{i\theta}) \in \H^2 (\DD)$ (on the
other hand, the latter follows from \cite[Theorem 16]{D07}) and
$g(e^{i\theta}) = \lim_{r \to 1-0} g(r e^{i \theta}) $ where $g(z) =
z \frac{du(z)}{dz}$, $z \in \DD$. By \cite[Problem II.5 (a)]{G81},
$|g(x)| \leq \| g \|_{\L^2} (1-|x|^2)^{-1/2}$ for $x \in (0,1) $ and
therefore, for $x \in (1/2,1)$,
\begin{equation} \label{e ux<c2}
\left|\frac{d\ux (x)}{dx} \right| \leq \| g \|_{\L^2} |x|^{-1}
(1-|x|^2)^{-1/2} \leq C_1 (1-x)^{-1/2} .
\end{equation}

By \cite[Theorem 5.8 (ii)]{W87},
the operator $B_{\one}$ defined by $B_{\one} \ux := \b [\ux]$
on the domain
\[
\dom(B_{\one}) := \{ \ux \in \dom (B_{\max}) : [\one, \ux]_1 =0 \},
\quad [\one, \ux]_1 := \lim_{x \to 1-0} p(x)\ux'(x),
\]
is self-adjoint. Note that the limit $[\one, \ux]_1$ exists for any
$\ux \in \dom(B_{\max})$ due to \cite[Theorem 3.10]{W87}).

It follows from (\ref{e ux<c2}) that, for any eigenvector
$u(e^{i\theta}) $ of $A_+$, its restriction $\ux$ belongs to $\dom
(B_{\one})$. Indeed, it was shown in the step (1) of the proof that
$\ux \in \dom (B_{\max})$. On the other hand, it follows from
(\ref{e ux<c2}) that
\[
[\one, \ux]_1 = \lim_{x \to 1-0} (1-x)^{1+1/\ep} (x+1)^{1-1/\ep} u'(x) =0 .
\]
So $\mu$ is an eigenvalue of the operator $B_{\one} = B_{\one}^*$.

It follows from (\ref{e ux<c2}) that
$\ux (x) = \ux (1/2) + \int_{1/2}^x \ux' (t) dt $ has a finite limit
as $x \to 1-0$ (this fact also follows from \cite[Theorem 16]{D07}).
Therefore,
\begin{gather}
(B_{\one} \ux,\ux)_{\L^2 ((0,1);w(x))} = -\int_0^1 (p(x)\ux'(x))' \overline{\ux(x)} dx = \int_0^1 p(x)|\ux'(x)|^2 dx  - \lim_{x \to 1-0} p(x) \ux' (x) \overline{\ux (x)} \\
= \int_0^1 p(x)|\ux'(x)|^2 dx \geq 0.
\end{gather}
Thus, $B_{\one} \geq 0$ and therefore $\mu \geq 0$.

Finally, note that $\ker (L_+) =0$ and therefore $\mu \neq 0$.
\end{proof}

\begin{remark} \label{r Weir}
For $\ep \in (0,2)$ such that $1/\ep \not \in \Z$, a slightly
different form of Proposition \ref{p real} was proved in
\cite[Theorem 2.3]{Weir1} by means of Proposition  \ref{p W1}, the
Frobenius theory, and the deep analysis of eigenvectors of the
corresponding recursion relation given in \cite[Sections 2 and
3]{D07}. Our proof that removes the condition $1/\ep \not \in \Z$,
is based on the description of $\dom (L)$ given in Section \ref{s
DE} and the Hardy spaces theory.
\end{remark}

\begin{theorem}
\label{basis} If $\ep \in (0,2)$, then the set of eigenvectors of
the operator $L$ does not form a Riesz basis in $\L^2 (-\pi,\pi)$.
\end{theorem}

\begin{proof}
Assume that the set  $\{ u_n \}_1^{\infty}$ of all (linearly
independent) eigenvectors of $L$ form a Riesz basis in $\L^2
(-\pi,\pi)$. Then Proposition \ref{p real} implies that $iL$ is
similar to a certain self-adjoint operator $Q$. That is, there
exists a bounded and boundedly invertible operator $S$ such that
$S\dom(Q) = \dom(L)$ and $iL = S Q S^{-1}$.

The spectral theorem for a self-adjoint operator implies that,
for arbitrary $u_0 \in \dom (L) (=X_2)$, the problem
\[
u_t + Lu=0, \quad u \mid _{t=0} = u_0 , \quad t \in \R,
\]
has a unique solution $u(\cdot,t)$ in the sense of \cite[Definition I.1.1 and Eq.~(I.1.2)]{K71}
(such solutions are sometimes called \emph{strong solutions}).
Moreover, this solution has the form $u(\cdot,t) = S e^{-itQ} S^{-1}
u_0 (\cdot)$. Therefore, for any $T>0$,
\begin{gather*}
u \in C([0,T]; X_2 ) \subset
C([0,T]; \W_{2p}^1 (-\pi,\pi) ), \\
u_t (\cdot,t) \in C([0,T]; \L^2 (-\pi,\pi)), \quad \text{and}  \quad
L u (\cdot,t) \in C([0,T]; \L^2 (-\pi,\pi)).
\end{gather*}
It is easy to see that $u$ is a generalized solution of
(\ref{eqno6}), (\ref{eqno7}) in the sense of Definitions \ref{d
gs1}-\ref{d gs3}. Since $e^{-itQ}$ is a unitary operator,
\[
\| u (\cdot,t)\|_{\L^2 (-\pi,\pi)} \leq \| S \| \, \| S^{-1} \| \,
\|u_0\|_{\L^2 (-\pi,\pi)} , \quad t \in \R.
\]

Hence, for any $T>0$, we have
\[
\int_{0}^{T} \| u (\cdot,t) \|_{\L^2 (-\pi,\pi)}^2 dt \leq C T \|
u_0 \|_{\L^2 (-\pi,\pi)}^2,
\]
where $C= \| S \|^2 \, \| S^{-1} \|^2 < \infty$. The latter
contradicts Theorem \ref{instab} since $\dom (L)$ is dense in
$\W_{2p}^2(-\pi,\pi)$ and $\| u_0 \|_{\L^2 (-\pi,\pi)} \leq
\|u_0\|_{\W_{2p}^k(-\pi,\pi)}$.
\end{proof}

\begin{remark} \label{r EigCompl}
Arguments of Proposition \ref{p real} show that 
all eigenvalues of the operators $L_\pm$ are simple.
Combining this with Theorems \ref{t Lcomp} and \ref{basis},
one can show that (under the assumption $\ep \in (0,2)$) \emph{the system 
$\{ u_n \}_1^\infty $ of eigenvectors of $L$ is complete in
$\L^2 (-\pi,\pi)$, but does not form a Riesz basis in
$\L^2 (-\pi,\pi)$.}

Indeed, let us show that all eigenvalues of $L_+$ are simple  
in the case $\ep \in (1,2)$.
Assume that $u_1(e^{i\theta}) \in \H^{2,0} (\DD) \setminus \{ 0 \}$ 
is a generalized first order eigenvector of $L_+$, i.e., $(L_+ - \lambda I) u_1 = u$, where 
$u(e^{i\theta}) \in \H^{2,0} (\DD) \setminus \{ 0 \}$
and $L_+ u = \lambda u$. 
Consider the restrictions $\ux$, $\ux_1$ of
the functions $u$ and $u_1$ on the interval $[0,1) \subset \DD$. 
It follows from the proof of Proposition \ref{p real} that 
$ \ux, \, \ux_1 \in \dom (B_{\one}) \subset \L^2 ((0,1);w)$. 
On the other hand, computations analogous to that of \cite[Lemma 2.1 and Theorem 2.3]{Weir1} show that 
$ \b[\ux_1] - \mu \, \ux_1 = - \frac{2 i}{\ep} \ux $  
with $\mu= - \frac{2 i \, \lambda}{\ep}$.
Therefore $\ux_1$ is a generalized eigenvector of the self-adjoint operator $B_{\one}$, a contradiction. The proof for 
the case $\ep \in (0,1]$ is similar. 
\end{remark}

We would like to note that the linear partial differential equation
(\ref{heat PDE}) is an interesting example when the nature of
explosive blow-up and instability of solutions has its roots not in
location of the eigenvalues but in geometric properties of the
eigenfunctions.

\section{Further discussion}
When eigenfunctions related to neutrally stable eigenvalues of some
linearized problem form the complete set, representation of a
solution of the nonlinear problem as a series of these
eigenfunctions is one of general approaches to the nonlinear
stability problem. The lack of a basis property of the eigenfunction
set is an obstacle for the applicability of this particular method.

Due to the ill-posed nature of the forward-backward heat equation
all eigenmodes are linearly unstable \cite{Benilov1} and it is
common to use the smoothing effect of the surface tension to
stabilize them . The lubrication approximation that takes into
account the influence of the capillary effects and/or surface
tension leads to the initial value problem for the fourth order
nonlinear partial differential equation \cite{Chapman}. Some
stability properties of its linearization were studied in
\cite{Benilov2,Benilov3,Brien}. They came to the conclusion that
almost all but some first modes are getting stable even if the
surface tension is relatively weak.

We would also like to mention that the main assumption about the 
parameter range  $|\varepsilon| < 2$ comes naturally from the theory
of mixed type equations and for the case when $|\varepsilon| > 2$
all properties of this backward forward heat equation can be changed
significantly.

Let us consider the equation
\begin{equation}
\label{mixed} k(x,t) \, u_{tt}+\alpha(t,x) \, u_t + \Delta u = 0, \,
\ x \in \Omega,\ t>0
\end{equation}
where the coefficient $k(x,t)$ can change sign in the domain where
the operator is considered. So equation (\ref{mixed}) is an
equation of the mixed type, i.e. it is of the same type as the
well-known Tricomi equation. On the lateral boundary of the cylinder
$\Omega\times (0,T)$ we pose the Dirichlet boundary condition and
there are two additional boundary condition on the lower and upper
base of the cylinder:
$$
u|_{t=T}=0,\ \ u_t|_{S^+}=0, \ \ u_t|_{S^-}=0,\ \ S^+=\{(0,x):\
k(x,0)>0\}, \ \ S^-=\{(T,x):\ k(x,T)<0\}.
$$
This boundary value problem and close problems were studied by many
authors (see, for instance, \cite{Vrag1, Vrag2}). It was
demonstrated that the condition
$$
\alpha - \frac{k_t}{2} \geq \delta_0 > 0\ \ \forall (x,t),
$$
where $\delta_0$ is a positive constant, ensures the existence of
generalized solutions to the above-described boundary value problem.
Stronger conditions of the type
$$
\alpha-\frac{|k_t|(2k-1)}{2}\geq \delta_0>0\ \ \forall (x,t)
$$
ensure existence of smooth solutions and uniqueness of generalized
solutions. The existence of solutions of  non-linear
forward-backward heat equations was studied by Hollig \cite{Hollig}
and by Pyatkov \cite{Pyatkov}. Among last results devoted to the
nonlinear forward-backward parabolic problems we would like to
mention Kuznecov papers \cite{Kuznecov1, Kuznecov2}.

{\bf Acknowledgement.} The authors 
thank A.~Burchard and E.B.~Davies for useful
comments and  discussions. The research of 
M.~Chugunova is
supported by the NSERC Postdoctoral Fellowship. I.M.~Karabash  would like to thank P.~Binding for the hospitality 
of the University of Calgary.

\appendix
\section{Proof of Corollary \ref{closed} using Galerkin method. \label{a A}}



\begin{proof}[The second proof of Corollary \ref{closed}.]

Let $\{\omega_j\}_{j=1}^\infty$ be a basis for the Hilbert space
$H=\{h \in \W_{2p}^2(-\pi,\pi):\
\int_{-\pi}^{\pi}h(\theta)\,d\theta=0\}$. Find functions $\varphi_j$
such that $\varphi_{j\theta}=\omega_j$,
$\int_{-\pi}^{\pi}\varphi_j(\theta)\,d\theta=0$. We look for an
approximate solution to equation (\ref{diffeq}) in the form
$$
h_n=\sum\limits_{j=1}^n c_{jn}\varphi_j,
$$
where the constants $c_{jn}$ are determined from the system of
algebraic equations
\begin{equation}
\label{eqno2} (Lh_{n},\omega_j)=(f,\omega_j),\ \ j=1,2,\ldots,n,
\end{equation}
(the brackets denote the inner product in $\L^2(-\pi,\pi)$, i.e.,
$(h,v)=\int_{-\pi}^{\pi}h(\theta)v(\theta)\,d\theta$).

Multiply (\ref{eqno2}) by $c_{jn}$ and summarize the equalities
obtained. We arrive   at the relation
$$
(Lh_{n},h_{n\theta})=(f,h_{n\theta}).
$$
Integrating by parts we derive the estimate
\begin{equation}
\label{eqno3} \|h_{n\theta}\|_{\L^2(-\pi,\pi)}\leq
c\|f\|_{\L^2(-\pi,\pi)},
\end{equation}
 where
$c$ is a constant independent of $n$. This estimate implies that the
system (\ref{eqno2}) is solvable. Note that there exists a constant
$c_1$ independent of $n$ such that
\begin{equation}
\label{eqno4} \|h_n\|_{\L^2(-\pi,\pi)}\leq c
\|h_{n\theta}\|_{\L^2(-\pi,\pi)}
\end{equation}

 From (\ref{eqno3}), (\ref{eqno4}) we conclude that there exists a subsequence $h_{n_k}$ and a function
$h \in \W_2^1(-\pi,\pi)$,  $h(-\pi)=h(\pi)$ and
$\int_{-\pi}^{\pi}h(\theta)\,d\theta=0$, such that
\begin{equation}
\label{eqno5} h_{n_k}\to h \ \ \textrm{in} \ \L^2(-\pi,\pi), h_{n_k
\theta}\to h_{\theta} \ \ \textrm{weakly in} \ \L^2(-\pi,\pi).
\end{equation}

Multiply (\ref{eqno2}) with $n=n_k$ by constants $\alpha_j$ $(1\leq
j\leq m\leq n_k)$ and sum the results. Fix $m$ assuming that
$n_k\geq m$. We infer
$$
-\varepsilon(\sin \theta
h_{n_k\theta},\omega_\theta)+(h_{n_k\theta},
\omega)=(f(\theta),\omega), \ \ \omega=\sum\limits_{j=1}^m
\alpha_j\omega_j.
$$
Passing to the limit as $k\to \infty$ we arrive at the equality
\begin{equation}
\label{eqno2.8} -\varepsilon(\sin \theta
h_{\theta},\omega_\theta)+(h_{\theta}, \omega)=(f(\theta),\omega), \
\ \omega=\sum\limits_{j=1}^m \alpha_j\omega_j.
\end{equation}
The functions $\omega$ of the form $\omega=\sum\limits_{j=1}^m
\alpha_j\omega_j$ are dense in $H$ and thus (\ref{eqno2.8}) holds
for all functions in $H$. Due to our condition for the function $f$,
we can see that (\ref{eqno2.8}) also holds for all functions of the
form $\omega+c$ ($c$ is an arbitrary constant) and therefore for all
functions in $\W_{2p}^2(-\pi,\pi)$. In particular, it holds for
$\omega \in C_0^{\infty}(-\pi,\pi)$. From the definition of the
generalized derivative (in the Sobolev sense) we have that there
exist the generalized  derivative $(\sin \theta h_\theta)_\theta$
and
$$
\varepsilon(\sin \theta h_\theta)_\theta=(f-h_\theta)\in
\L^2(-\pi,\pi).
$$
Thereby, $\sin \theta h_\theta\in \W_2^1(-\pi,\pi)$. Integrating by
parts in (\ref{eqno2.8}) we obtain that the equation (\ref{diffeq})
is satisfied almost everywhere on $(-\pi,\pi)$. We have proven that
$L$  is an isomorphism of $H_2^0$ onto $\L_p^2(-\pi,\pi)$. The
remaining assertions more or less obvious.

\begin{remark} We can take the set $\{\sin j\theta,\cos
j\theta\}_{j=1}^\infty$ rather than an abstract basis
$\{\omega_j\}$.
\end{remark}

In view of the embedding theorems, the operator
$L^{-1}:\L_p^2(-\pi,\pi)\to \L_p^2(-\pi,\pi)$ is compact and thus
the spectrum of $L$ is discrete with the only accumulation point
$\infty$. Assume the contrary that there exist $\lambda\in \rn$ such
that
$$
L h =\lambda h,\ \ \ h\in H_2^0.
$$
Multiply the equation  by $h$ and  integrate the result over $(-\pi,
\pi)$. Integrating by parts and taking the real part, we arrive at
the inequality
$$
\|h_{\theta}\|_{\L^2(-\pi,\pi)}^2\leq 0.
$$
 We have used the equality
$\Re(h, h_{\theta})=0$. Thus, $h \equiv 0$.
\end{proof}

\begin{theorem}
Under the condition $1-\varepsilon(k+1/2)>0$, for every $f\in
\W_{2p}^{k}(-\pi,\pi)$ there exists a unique solution to the problem
(2.1) such that $u\in \W_{2p}^{k+1}(-\pi,\pi)$, $\sin \theta
h^{(k+1)}\in \W_2^1(-\pi,\pi)$.
\end{theorem}

\begin{proof}
We can use the same arguments as in the second proof of Corollary
\ref{closed}. It is not difficult to show that there exist set of
constants $\lambda_i>0$ such that in the equivalent inner product in
the space $\W_2^k(-\pi,\pi)$
$$
(h,v)_k=\sum\limits_{j=0}^k \lambda_j(h^{(j)}, v^{(j)})
$$
we have the inequality
$$
(Lh,h_\theta)_k\geq \delta_0 \|h\|_{\W_2^{k+1}(-\pi,\pi)}^2, \ \
\forall h\in \W_{2p}^{k+2}(-\pi,\pi),
$$
where the constant $\delta_0>0$ is independent of $h$. Next we apply
the same arguments as those in the theorem 1 but we use the inner
product $(h,v)_k$ rather than the inner product $(h,v)$ in
$\L^2(-\pi,\pi)$. So the Galerkin method is applicable here.
\end{proof}

\begin{remark}
It is also possible to use some functional arguments based on the
Hahn-Banach   theorem.
\end{remark}


\begin{thebibliography}{50}
\bibitem{Heun} F. M. Arscott, Heun's equation, in: A. Ronveaux (Ed.), {\em Heun's
Differential Equations, Oxford University Press, Oxford,}  21 -- 24,
(1995)

\bibitem{Benilov1} E. S. Benilov, S. B. G. O'Brien, and I. A. Sazonov,
"A new type of instability: explosive disturbances in a liquid fild
inside a rotating horizontal cylinder", J. Fluid Mech. {\bf 497},
201--224 (2003)

\bibitem{Benilov2} E. S. Benilov, M. S. Benilov, N. Kopteva, "Steady rimming flows
with surface tension", J. Fluid Mech. {\bf 597}, 91--118 (2008)

\bibitem{Benilov3} E. S. Benilov, N. Kopteva and S. B. G. O'Brien, "Does surface
tension stabilise liquid films inside a rotating horizontal
cylinder", Q. J. Mech. Appl. Math. {\bf 58}, 158--200 (2005)

\bibitem{PT-sym} L. Boulton, M. Levitin, and M. Marletta, "A PT-symmetric periodic problem with
boundary and interior singularities", arXiv:0801.0172v1 [math.SP]
(2008)

\bibitem{Brien}  S. B. G. O'Brien, "A mechanism for two dimensional
instabilities in rimming flow", Q. Appl. Maths. {\bf 60}, 283--300
(2002)

\bibitem{light model} S. B. G. O'Brien, "A model for the coating of cylindrical light
bulbs," in Progress in Industrial Mathematics at ECMI {\bf 98}, p.48
(1998)

\bibitem{Chapman}  S. J. Chapman, "Subcritical transition in channel flow",
J. Fluid Mech. {\bf 451}, 35--97 (2002)

\bibitem{ChugPel} M. Chugunova, D. Pelinovsky, "Spectrum of a non-self-adjoint operator
associated with the periodic heat equation", J.Math.Anal.Appl.,
doi:10.1016/j.jmaa.2007.12.036, (2008)

\bibitem{ChugStr} M. Chugunova, V. Strauss, "Factorization of the Indefinite Convection-Diffusion
Operator", to appear in Math. Reports Acad. Sci. Royal Soc. Canada,
(2008)

\bibitem{D07} E. B. Davies, "An indefinite convection-diffusion operator",  LMS J. Comput. Math.  {\bf 10},
288--306 (2007)

\bibitem{G81} J. B. Garnett, {\it Bounded analytic functions,} Academic Press, Inc., New York-London,
(1981)

\bibitem{GohMar} I. C. Gohberg, A. S. Markus, "On some relations between eigenvalues
and matrix elements of linear operators", Mat. sb. 64(106) (1964),
481-496. English transl. in: Amer. Math. Soc. Transl. (2), 52(1966)

\bibitem{GohKrein} I. Gohberg and M. G. Krein, {\em Introduction to the Theory of Linear
Non-selfadjoint Operators}, Vol. {\bf 18} (AMS Translations,
Providence, 1969)

\bibitem{Hollig} K. Hollig, "Existence of Infinitely Many Solutions for a Forward Backward
Heat Equation",  Transactions of the American Mathematical Society,
{\bf 278}, No. 1., 299--316 (1983)


\bibitem{K71} S.G. Krein,  {\it Linear differential equations in Banach space}, AMS,
(1971)

\bibitem{Kuznecov1} I. V. Kuznetsov, "Entropy solutions to a second order
forward-backward parabolic differential equation" (Russian, English)
Sib. Mat. Zh. {\bf 46}, N 3, 594--619 (2005); translation in Sib.
Math. J. {\bf 46}, N 3, 467--488 (2005)

\bibitem{Kuznecov2} I. V. Kuznetsov, "Entropy solutions to a second-order
forward-backward parabolic equation" (English. Russian original)
Dokl. Math. {\bf 72}, N 2, 716--717 (2005); translation from Dokl.
Akad. Nauk, Ross. Akad. Nauk {\bf 404}, N 4, 443--445 (2005)


\bibitem{LadSolUr} O.A. Ladyzhenskaya, V.A. Solonnikov, N.N. Ural'ceva, {\it Linear and
Quasilinear Equations of Parabolic Type}, Translations of
Mathematical Monographs, Vol. 23, American Mathematical Society,
Providence, RI, (1968)


\bibitem{LanTret} H. Langer, C. Tretter, "A Krein space approach to PT-symmetry",
Czechoslovak J. Phys. {\bf 54}, N 10, 1113.1120 (2004)

\bibitem{M72} B. Muckenhoupt, "Hardy's inequality with weights", Studia Math. {\bf 34}, 31--38 (1972)

\bibitem{Pyatkov} S. G. Pyatkov, "Solvability of initial-boundary value
problems for a nonlinear parabolic equation with changing time
direction", Preprint, {\bf 16}, Novosibirsk, Institute of
Mathematics (in Russian), (1987)

\bibitem{S84}  E. Sawyer,  "Weighted Lebesgue and Lorentz Norm Inequalities for the Hardy Operator",
Transactions of AMS, {\bf 281}, No. 1., 329 --337 (1984)

\bibitem{Shin1} K C Shin, "On the reality of the eigenvalues for a class of PT -
symmetric oscillators", Comm. Math. Phys. {\bf 229}, N 3, 543 -- 564
(2002)

\bibitem{Shin2} K C Shin, "Eigenvalues of PT -symmetric oscillators with
polynomial potentials", J. Phys. A {\bf 38}, N 27, 6147 -- 6166
(2005)


\bibitem{rotmodel} J. L. Throne, J. Gianchandani, "Reactive rotational molding,"
Polym. Eng. Sci. {\bf 20}, p.899 (1980)


\bibitem{Vrag1} V. N. Vragov  {\em Boundary Value Problems for
Nonclassical Equations of Mathematical Physics}, Novosibirsk State
University, Novosibirsk (in Russian)(1983)

\bibitem{Vrag2} V. N. Vragov, A. I. Kozhanov, S. G. Pyatkov, S. N.
Glazatov, "On the theory of nonclassical equations of mathematical
 physics", {\it Conditionally Well--Posed Problems.} TVP/TSP, Utrecht,
299--321, (1993)


\bibitem{W87}  J.Weidmann, {\it Spectral theory of ordinary differential operators},
Lecture Notes in Mathematics, {\bf 1258} Springer-Verlag, Berlin,
(1987)

\bibitem{Weir1} Weir, John, "An Indefinite Convection-Diffusion Operator With Real Spectrum",   arXiv:0711.1371v1
[math.SP] (2008)

\end{thebibliography}
\end{document}